\documentclass[12pt]{amsart}

\usepackage{amssymb}
\usepackage{latexsym}
\usepackage{graphicx}
\usepackage{tikz}
\usetikzlibrary{arrows}

\theoremstyle{plain}
\newtheorem{prop}{Property}[section]
\newtheorem{theorem}[prop]{Theorem}

\newtheorem{corollary}[prop]{Corollary}
\newtheorem{lemma}[prop]{Lemma}
\newtheorem{conjecture}{Conjecture}

\theoremstyle{remark}

\newcommand{\R}{{\mathbb R}}
\newcommand{\Z}{{\mathbb Z}}
\newcommand{\N}{{\mathbb N}}

\newcommand{\beeq}{\begin{eqnarray*}}
\newcommand{\eneq}{\end{eqnarray*}}

\newcommand{\be}{\begin{equation}}
\newcommand{\ee}{\end{equation}}

\title{On Doubling and Volume: Chains}
 \author{Gregory A. Freiman}
 \address{ 
The Raymond and Beverly Sackler Faculty of Exact Sciences\\
School of Mathematical Sciences\\
Tel Aviv University}
\email{grisha@post.tau.ac.il}

 \author{Oriol Serra}
 \address{Department of Mathematics, Universitat Polit\`ecnica de Catalunya and\\ 
 Barcelona Graduate School of Mathematics, Barcelona}
 \email{oriol.serra@upc.edu}
 \thanks{O.\,Serra was supported by the Spanish Ministerio de Econom\'ia y Competitividad under project MTM2014-54745-P}
 \date{}

\begin{document}

  \begin{abstract}  The well--known Freiman--Ruzsa Theorem provides a structural description of a set $A$ of integers with $|2A|\le c|A|$ as a  subset of a $d$--dimensional arithmetic progression $P$ with $|P|\le c'|A|$, where $d$ and $c'$ depend only on $c$.  The estimation of the constants $d$ and $c'$ involved in the statement  has been the object of intense research. Freiman conjectured in 2008 a formula for  the largest volume of such a set. In this paper we prove the conjecture for a general class of sets called chains.  
  \end{abstract}

   \maketitle
   
\section{Introduction}

The Freiman--Ruzsa theorem giving the structure of sets of integers with small doubling is one of the deep results in  Additive Number Theory:

\begin{theorem}[Freiman--Ruzsa] Let $A$ be a finite set of integers. If $|2A|\le c|A|$ then there are constants $d, c'$ depending only on $c$ such that $A$ is contained in a multidimensional arithmetic progression $P$ with dimension $d$ and cardinality $|P|\le c' |A|$.
\end{theorem}

The estimation of the constants $d$ and $c'$ involved in the statement has been the object of a long series of papers. From the first proof of Freiman \cite{Freiman87} and Bilu \cite{Bilu99} one can obtain a fourth exponential dependence of $c'$ on $c$. The proof by Ruzsa \cite{Ruzsa94}, which provided an estimation of the form $c'\le  exp (c^{c^{c}})$, was subsequently refined by Chang \cite{Chang2002}, giving $d\le c^{2+o(1)}$ and $c'\le \exp(c^2+o(1))$, further improved by Sanders to $d\le c^{4/7+o(1)}$ and $c'\le \exp(c^{4/7+o(1)})$ and eventually brought to its essentially best values $d\le c^{1+k(\log c)^{-1/2}}$ and $c'\le \exp(c^{1+k(\log c)^{-1/2}})$, $k$ an absolute constant,  by Schoen \cite{Schoen2011}.

In a conference in Toronto in 2008, Freiman proposed a precise formula for the largest possible volume of a set $A$ with given doubling $T=|2A|$ in terms of a specific parametrization of the value of $T$, see e.g. Freiman \cite{Freiman2014}.  We next recall some definitions in order to give this conjectured formula for the maximum volume.

 Let $G, G'$ be two abelian groups. Two finite sets $A\subset G$ and $B\subset G'$ are {\it Freiman isomorphic} ($F$--isomorphic for short) if there is a bijection $\phi:A\to B$ such that, for every $x,y,z,t\in G$,
$$
x+y=z+t \; \Leftrightarrow \phi(x)+\phi(y)=\phi(z)+\phi(t),
$$ 
in which case we write $A\cong_F B$. 

A set $A$ is in {\it normal form} if $\min (A)=0$ and $\gcd (A)=1$. Every set $A$ is $F$--isomorphic to $\tilde{A}=(A-\min (A))/\gcd(A)$ which is in normal form, and $\tilde{A}$ is the {\it normalization} of $A$. If $\min(A)=0$, then the {\it reflexion} of $A$ is defined as $A^-=-A+\max(A)$.

 The {\it additive dimension} $\dim (A)$ of a set $A\subset G$ is the largest $d$ such that there is a  set $B\subset \Z^d$ not contained in a hyperplane of $\Z^d$ which is $F$--isomorphic to $A$. 
 
 The {\it volume} $vol (A)$ of a $d$--dimensional set $A$ is the minimum cardinality, among all sets    $B\subset \Z^d$ which are $F$--isomorphic to $A$, of the convex hull of $B$. In particular, if $A$ is a $1$--dimensional set in normal form, then $$vol(A)=\max(A)+1.$$
 
 We are interested in obtaining upper bounds for the volume of a set $A$ of integers in terms of its cardinality $|A|$ and the cardinality of its doubling $|2A|$. We denote by
$$
vol(k,T)=\max \{vol (A): A\subset \N \;, |A|=k, |2A|=T \},
$$
the maximum value of the volume of a  set $A$ of integers among all  sets with cardinality $k$ and doubling $T$.

 A set $A$ is {\it extremal} if  $vol(A)=vol(|A|,|2A|)$.  The following conjecture is stated in Freiman \cite{Freiman2014}.

\begin{conjecture}[Freiman]\label{conj:main} Let $A$ be a   set of integers with cardinality  $k=|A|\ge 4$. If
\begin{equation}\label{eq:2}
|2A|=ck-{c+1\choose 2}+b+2,
\end{equation}
where $2\le c\le k-2$ and   $1\le b\le  k-c-1$, then
\begin{equation}\label{eq:max}
vol (A)\le 2^{c-2}(k-c+b+1)+1.
\end{equation}
Moreover the inequality is tight and it is reached by $1$--dimensional sets.
\end{conjecture}

Conjecture \ref{conj:main}  is proved for   $2k-1\le |2A|\le 3k-4$. For these values it is known that $A$ is $1$--dimensional and that its largest element is at most $k+b-1$, the bound being tight. Moreover the structure of  extremal sets with maximum element $k+b-1$ can be described in detail (see Freiman \cite{Freiman2009} and Section \ref{sec:3k-4}).

According to the notation in Conjecture \ref{conj:main}, given $k$ and $T\in [2k-1,{k\choose 2}+2]$ there are uniquely defined 
$$
c=c(k,T)\; \text{and}\; b=b(k,T)
$$
subject to the boundary conditions $2\le c\le k-2$ and $1\le b\le k-c-1$, or $b=0$ and $c=2$, such that $T$ can be expressed as the right--hand side of \eqref{eq:2}.   If $A$ has cardinality $k$ and doubling $|2A|=T$ then we call $c=c(k,T)$ the {\it doubling constant} of $A$. Thus, for example, the doubling constant is $2$ if $|2A|=3|A|-4$ and it is $3$ if $|2A|=3|A|-3$, according to the structural change on $A$ on these values of its doubling. We also denote by
\begin{equation}\label{eq:mudef}
\mu (k,T)= 2^{c-2}(k-c+b+1),
\end{equation}
the conjectured maximum volume minus one of a set with cardinality $k$ and doubling $T$.

The main result of this paper proves Conjecture \ref{conj:main} for a general class of sets called chains. Roughly speaking, chains are sets with maximum volume among all sets which can be obtained by a sequence of sets of the same nature starting with an arithmetic progression of length three (see Section \ref{sec:chains} for precise definitions). 

Moreover, the structure of chains can be described.  We denote by $\Phi$ a family of operators on sets which either add to a set $X$ in normal form the element $2\max (X)$ or multiply all of its elements by $2$ and add an odd number (see the end of Section \ref{sec:lb} for a precise definition of the family $\Phi$ .) The main result in this paper is the following one:

\begin{theorem}\label{thm:chain} Let $A$ be a chain with $k=|A|$ and $2k-1\le |2A|\le {k\choose 2}+2$. Then
$$
vol (A)= \mu (k,T)+1.
$$
Moreover,   there are subchains $B\subseteq B'\subseteq A$ with $|2B|\le 3|B|-4$ and an integer $s\ge 0$ such that
$$
A\cong_F \phi_{s}\phi_{s-1} \cdots \phi_1(B'),
$$
where each $\phi_i\in \Phi$ and  $|B'|\le |B|+1$.
\end{theorem}

As we have already mentioned,  the structure of $B$ in  Theorem \ref{thm:chain}  is already well--known by the so--called $(3k-4)$--Theorem, which we recall in Section \ref{sec:3k-4}.
Therefore  Theorem \ref{thm:chain} gives a precise structural description of chains (the case in which $|B'|=|B|+1$ is clarified as well in Section \ref{sec:chains}.)   One remarkable feature of Theorem  \ref{thm:chain} is that it holds for chains $A$ with  $|2A|=c|A|$ for the doubling constant $c$ up to  $(|A|-1)/2$. 

One consequence of Theorem \ref{thm:chain} is to prove Conjecture \ref{conj:main} for the class ${\mathcal C}$ of chains. In particular, for every $k\ge 4$ and $T\in [2k-1,{k\choose 2}+2]$, we have
\begin{equation}\label{eq:vmu}
vol (k,T)\ge \mu (k,T)+1.
\end{equation}

The paper is organized as follows. We give general terminology and basic results in Section \ref{sec:term}. The $(3k-4)$--Theorem of Freiman is recalled in Section \ref{sec:3k-4}.  In Section \ref{sec:lb} we give a construction of sets with volume $\mu (k,T)+1$ for every suitable value of $T$, giving the lower bound \eqref{eq:vmu} for $vol (k,T)$.   The definition of chain and the  lemmas leading to the proof of the main result  are given in Section \ref{sec:chains}.   The proof of Theorem \ref{thm:chain} is given in Section \ref{sec:proof}. The proof is based on structural properties of chains, and the nature of the result calls for elementary methods. We conclude the paper with some final remarks in Section \ref{sec:final}.

\section{Notation and preliminary results}\label{sec:term}

A set $A=\{a_0<a_1<\cdots <a_{k-1}\}$ of integers  is in {\it normal} form if $a_0=0$ and $\gcd (A)=1$.  
The convex hull of a set of integers is the smallest interval containing it. We call a {\it hole} of $A$ an element in its convex hull not contained in $A$.

A $d$--progression is a set of the form $\{a,a+d,a+2d,\ldots ,a+(k-1)d\}$ for some $d\ge 1$ and $k\ge 1$. We note that, with this definition, a singleton set  $\{a\}$ is a $d$--progression for each $d\ge 1$.

We parametrize the values of the doubling $T=T(A)=|2A|$ of a set $A$ of integers with cardinality $k=|A|$ as follows. For each $k\ge 4$ and each $c\in [2,k-2]$ we denote by $I_c$ the integer interval
$$
I_{c,k}=ck-{c+1\choose 2}+2+\left[ 1, k-c-1\right],
$$
including the value $2k-1$ in $I_{2,k}$, so that their disjoint union
$$
\bigcup_{c=2}^{k-1} I_{c,k}=\left[2k-1, {k\choose 2}+2\right],
$$
is a range of values of  $T(A)$. Given $k$ and $T$ we denote by $c(k,T)$ the value of $c$ for which $T\in I_{c,k}$ and we write
\begin{equation}\label{eq:kbdef}
T=ck-{c+1\choose 2}+b+2,
\end{equation}
where $b=b(k,T)\in [1,k-c-1]$. We call $c=c(k,T)$ the  doubling constant of the set $A$.

We next recall a key result on the additive  dimension of a set.  Let  $A=\{a_1,\ldots ,a_k\}\subset G$,  be a finite subset in an abelian group $G$. Let $e_1,\ldots ,e_k$ denote a basis of the $k$--dimensional real vector space $\R^k$. To each  relation of the form $a_i+a_j=a_r+a_s$ satisfied by the elements of $A$ we associate the vector $e_i+e_j-e_r-e_s\in \R^k$. Let us denote by $\lambda (A)$ the dimension of the subspace of $\R^k$ generated by all these vectors. This dimension is related to the additive dimension of $A$ by the following Theorem of Konyagin and Lev  \cite{KL2000}. 

\begin{theorem}[Konyagin, Lev] \label{thm:kl}The additive dimension of a set $A$ with cardinality $k$ is
$$
dim (A)=k-1-\lambda (A).
$$
\end{theorem}

For example, the set  $\{a_1,a_2,a_3,a_4\}=\{0,1,2,4\}$ has dimension one since it contains the relations $a_1+a_3=2a_2$ and $a_1+a_4=2a_3$ which correspond to two independent vectors in $\R^4$, while $\{a_1,a_2,a_3,a_4'\}=\{0,1,2,5\}$ is $2$--dimensional.

We will often use the following  simple consequence of Theorem \ref{thm:kl}. 

\begin{corollary} Let $A$ be a $1$--dimensional set in normal form and $x>\max (A)$. Then $A\cup \{x\}$ is $1$--dimensional if and only if  $x\in 2A-A$.
\end{corollary}

\begin{proof} According to Theorem \ref{thm:kl}, the set $A\cup \{x\}$ is $1$--dimensional if and only if  $x$ is involved in an additive relation with the elements of $A$, that is, $a+x=a'+a''$  for some $a,a',a''\in A$. 	
	\end{proof}

\section{Stable sets and the $(3k-4)$--Theorem}\label{sec:3k-4}

We recall in this Section the $(3k-4)$--Theorem of Freiman giving the structure of extremal sets whose doubling is at most $3k-4$. We first need some definitions.

By a segment we mean a set of consecutive integers, denoted by $[a,a+k-1]=\{a,a+1,\ldots,a+k-1\}.$ The length of a set $A$ is
$$
\ell (A)=\max(A)-\min (A)+1.
$$
Given two sets $A, B$ in normal form, we denote their concatenation by
$$
A\circ B=A\cup (\max (A)+B).
$$
We thus have
\begin{align*}
\ell (A\circ B)&=\ell (A)+\ell (B)-1,\\
|A\circ B|&=|A|+|B|-1.
\end{align*}

A set $A=\{a_0= 0<a_1<\cdots <a_{k-1}\}$, $k\ge 2$, of integers is {\it stable} if
$$
2A\cap [0,a_{k-1}]= A  \;\;  \text{ and }\;\;  \{1,a_{k-1}-1\}\cap A=\emptyset.
$$
By convention we  say that $A=\{0\}$ is also a  stable set.
We say that a set $A$ is {\it right} stable if its reflexion $A^-=-A+a_{k-1}$ is stable. The typical examples of stable sets are $d$--progressions with $d\ge 2$, which are also right--stable. Actually, a stable set is always a union of $d$--progressions with difference $d=a_1$. For example, $\{0,4,5,8,9,12\}$ is stable (but not right--stable) and $\{0, 2, 3,6\}$ is right--stable (but not stable).

In fact, stability concerns the property that the doubling of $A$ leaves $A$ invariant in the interval $[0,\max(A)]$. We remark that, in our current definition of stable sets with more than one element,  we  impose the additional property that $A$ does not contain $1$ (so that  $A$ is not a segment)  and $A$ does not contain $a_{k-1}-1$. One reason to include this additional condition to the definition of stable sets is   the following simple Lemma, which gives the maximum density of initial segments of an stable set.

\begin{lemma}\label{lem:hole} Let   $A=\{0=a_0<a_1\cdots<a_{k-1}\}$ be a stable set. For each $x\in [0,a_{k-1}]$ we have
		$$
		|A\cap [0,x]|\le \left\lceil \frac{x+1}{2} \right\rceil.
		$$		
	\end{lemma}

\begin{proof} Let $A(x)=|A\cap [0,x]|$. If $x$ is a hole in $A$ then, since $A$ is stable,  $x\not\in 2A$.  Hence, at most one among $i$ and $x-i$ belongs to $A$ for each $0\le i\le x$, and  $A(x)\le \lceil x/2 \rceil$.
	
	In particular, since $a_{k-1}-1$ is a hole in $A$, we have  $A(a_{k-1})\le A(a_{k-1}-1)+1\le   \left\lceil \frac{a_{k-1}+1}{2} \right\rceil$. 
	
	Suppose that $x\in A$, $x<a_{k-1}$, and let $y>x$ be the smallest hole in $A$. We have,
		$$
		A(x)=A(y)-(y-x-1)\le \left\lceil\frac{y}{2}\right\rceil-(y-x-1)\le \left\lceil \frac{x+1}{2} \right\rceil.
		$$
\end{proof}

We say that a stable set  $A$ is {\it dense} if $|A|=\lceil (\max(A)+1)/2\rceil$. Arithmetic progressions of difference two are examples of dense stable sets. The stable set $\{0\}$ is also dense. 

We shall use the following structural characterization from Freiman \cite{Freiman2009} of extremal sets with  doubling at most $3k-4$.

\begin{theorem}[Freiman] \label{thm:3k-4} Let $A\subset \Z$ be an extremal set in normal form with $|A|=k$ and 
	$$
	|2A|=2k-1+b, \; 0\le b\le k-3.
	$$
	Then
	\begin{equation}\label{eq:stable}
	A=A_1\circ  P \circ A_2,
	\end{equation}
	where
	\begin{enumerate}
		\item[(i)] $A_1$  is stable and $A_2$ is right stable,
		\item[(ii)] $P$ is a segment with $|P|\ge k-b$, and
		\item[(iii)] $\ell (A)=k+b$. In particular $A$ is $1$--dimensional.
	\end{enumerate}
	
	Moreover
	\begin{equation}\label{eq:stable2}
	2A=A_1\circ P' \circ A_2,
	\end{equation}
	where  $P'$ is a segment with $|P'|\ge |2A|-b$.
\end{theorem}

The canonical decomposition of a set $A$ under the conditions of Theorem \ref{thm:3k-4} described in \eqref{eq:stable} will be called the {\it stable decomposition} of $A$. It is uniquely defined by our definition of stable sets of cardinality at least two having a hole (non element) in the one before the last position.

We  observe that, for each $T\in [2k-1,3k-4]$ there is an extremal set $A$ with $k=|A|$ and $|2A|=T$. For example, for each  $b\in [1,k-3]$  the set $$\{0\}\cup [b+1,k+b-1],$$ is an extremal set with doubling $T=|2A|=2k-1+b$.

\section{A lower bound for the volume}\label{sec:lb}

We will show that
$$
vol (k,T)\ge \mu (k,T)+1,
$$
by describing, for each $k\ge 4$ and   $2k-1\le T\le {k\choose 2}+2$,  a family of normalized $1$--dimensional sets with cardinality $k$, doubling $T$ and maximum element  $\mu (k,T)$. 

We define the following operations on a normal set  $A$:
\begin{enumerate}
\item[(i)] $D(A)=A\cup \{2\max (A)\}$.
\item[(ii)] $D_x(A)=2\cdot A \cup \{x\}=\{2a: x\in A\}\cup \{x\}$, $x$ odd and $x\in 2A\setminus A$.
\end{enumerate}

\begin{lemma}\label{lem:dd} Let $A$ be a $1$--dimensional set in normal form with $T=|2A|$. If $\max (A)=\mu (k,T)$ and $x$ is an odd number in $2A\setminus A$, then both $D(A)$ and $D_x(A)$ are normal $1$--dimensional sets with
$$
\max (D(A))=\max (D_x(A))=\mu (k+1, T+k).
$$
\end{lemma}

\begin{proof}  Let $c=c(k,T)$ and $b=b(k,T)$ so that
$$
T=ck-{c+1\choose 2}+b+2.
$$

Let $B=D(A)$ and $a=\max (A)$. We  have
$$
|2B|=|2A\cup (2a+A)\cup \{4a\}|=|2A|+k,
$$
and
$$
\max (B)=2a=2^{c-1}((k+1)-(c+1)+b+1).
$$
Since 
$$
c(k+1,T+k)=c(k,T)+1 \; \text{and } b(k+1,T+k)=b(k,T),
$$
we have $\max (B)=\mu (k+1, T+k)$. In order to show that $B$ is one--dimensional we just observe  that the new additive relation  corresponding to the $3$--term arithmetic progression $\{0,a,2a\}$ is linearly independent from the existing ones in $A$. It follows that  $\lambda (B)=\lambda (A)+1$ and therefore, according to Theorem \ref{thm:kl}, $B$ remains one--dimensional.

Let now $C=D_x(A)$. Since $x$ is the only odd number in $C$ we have
$$
|2C|=|2A|+k,
$$
and, since $x\in 2A$, we also have  $\max (C)=2\max (A)=\mu(k+1, T+k)$. We clearly have $\lambda(2\cdot A)=\lambda (A)$ (the two sets are $F$--isomorphic.)  Since $x\in 2A$ we have a relation of the form $a_i+a_j=x$ for some $i,j$. Hence $2a_i+2a_j=2x$ gives a relation in $C$ involving $x\not \in A$, which is therefore linearly independent of the existing relations in $2\cdot A$. Hence $\lambda (C)=\lambda (A)+1$ and $C$ is one dimensional. We note that the only relation involving $x$ in $C$ must be of the form $2a_i+2a_j=2x$, since $x$ is odd. Thus, if $x\not\in 2A\setminus A$ then $\dim (C)=\dim (A)+1$.
\end{proof}

By using the above operators $D$ and $D_x$ one can construct normal $1$--dimensional sets $A$ with $k=|A|, T=|2A|$ and  $\max (A)=\mu (k,T)$ for all suitable values of $T$.

\begin{corollary} For each $k\ge 4$ and each $T\in [2k-1, {k\choose 2}+2]$ there is a normal $1$--dimensional set  $A$ with $k=|A|, T=|2A|$ and  $\max (A)=\mu (k,T)$
\end{corollary}

\begin{proof} By Theorem \ref{thm:3k-4} the statement holds for   every $k\ge 3$ and every $T\in [2k-1,3k-4]$.  By induction on $k\ge 4$ and $T$, there is a normal $1$--dimensional set $A'$ with cardinality $k-1$ with $\max (A')=\mu (k-1, T-k)$. By Lemma \ref{lem:dd}, $A=D(A')$ is a normal $1$--dimensional set with $|A|=k$ and $|2A|=T$ and   $\max (A)=\mu (k,T)$.
\end{proof}

In particular we have shown that the value for the maximum volume in Conjecture \ref{conj:main} is tight. 

\begin{corollary}\label{cor:ub}  Let $k\ge 4$  and $T\in [2k-1, {k\choose 2}+2]$. Then
$$
vol (k,T)\ge \mu (k,T)+1. 
$$
Equivalently,  every extremal set $A$ with cardinality $k$ and  maximum element $\max (A)\le \mu (k,T)$ satisfies
$$
|2A|\le T.
$$

\end{corollary}

The  main result in the paper essentially states that chains are obtained by iterate application of operators of the form $D$ and $D_x$ to some extremal set with doubling constant $c=2$. This the meaning of the family $\Phi$ in the statement of Theorem \ref{thm:chain}. The family $\Phi$ is defined as follows. Let $X$ be a finite set of integers and let $\tilde{X}$ be its normal form. Then, $\phi \in \Phi$ if it acts in one of the following four ways: 
\begin{enumerate}
	\item[(i)] $\phi (X)=\tilde{X}\cup 2\max(\tilde{X})$, or
	\item[(ii)] $\phi (X)=(\tilde{X})^-\cup 2\max(\tilde{X})$, or
	\item[(iii)] $\phi (X)=2\cdot \tilde{X}\cup \{x\}$ for some odd $x\in 2\tilde{X}\setminus \tilde{X}$, or
	\item[(iv)] $\phi (X)=2\cdot (\tilde{X})^-\cup \{x\}$ for some odd $x\in 2(\tilde{X})^-\setminus (\tilde{X})^-$.
	\end{enumerate}


\section{Chains}\label{sec:chains}

Theorem \ref{thm:chain} gives a structural characterization of a general class ${\mathcal C}$ of sets that we call chains. We denote by $[A]=[\min(A), \max (A)]$ the convex hull of a set of integers.

A set $A$ is a {\it chain} if  there is a sequence
$$
A_3\subset A_4\subset \cdots \subset A_k=A
$$
such that 
\begin{enumerate}
\item[(i)] $A_3\cong_F \{0,1,2\}$,
\item[(ii)] each $A_i$, $4\le  i \le k$, is a $1$--dimensional set with $i$ elements with doubling $|2A_i|\le {i\choose 2}+1$, and $A_i\cap [A_{i-1}]= A_{i-1}$,
\item[(iii)] $A_i$ has largest volume among all sets $B$ with the same cardinality and doubling as $A_i$ and such that $B\cap [A_{i-1}]= A_{i-1}$.
\end{enumerate}

For example, up to isomorphism, the class of chains in normal form with  $5$ elements consist of the following sets:

\begin{center}
\begin{tikzpicture}[scale=0.6] 

\node at  (4,8) {$A$};
\foreach \i in {0,...,8}
{
\foreach \j in {1,...,7}
{
\draw (\i,\j) circle(3pt);
}
}
\foreach \i in {0,1,2,3,4}
{
\draw[fill] (\i, 7) circle(3pt);
}
\foreach \i in {0,2,3,4,5}
{
\draw[fill] (\i, 6) circle(3pt);
}
\foreach \i in {0,2,4,5,6}
{
\draw[fill] (\i, 5) circle(3pt);
}
\foreach \i in {0,3,4,5,6}
{
\draw[fill] (\i, 4) circle(3pt);
}
\foreach \i in {0,2,3,4,6}
{
\draw[fill] (\i, 3) circle(3pt);
}
\foreach \i in {0,4,5,6,8}
{
\draw[fill] (\i, 2) circle(3pt);
}
\foreach \i in {0,4,6,7,8}
{
\draw[fill] (\i, 1) circle(3pt);
}
\foreach \i in {0,...,8}
{
\node at (\i,0) {\i};
}

\node at (10,8) {$|2A|$};
\node at (10,7) {$9$};
\node at (10,6) {$10$};
\node at (10,5) {$11$};
\node at (10,4) {$11$};
\node at (10,3) {$11$};
\node at (10,2) {$12$};
\node at (10,1) {$12$};

\end{tikzpicture}
\end{center}

For every chain $A$, the deletion of either its larger element or   its smaller element is also a chain.  The set $\{0,3,4,6,7,8\}$, according to Theorem \ref{thm:3k-4}, is an extremal set (it has doubling $14$ and largest element $8$) but it is not a chain: by removing the last element we obtain a two--dimensional set, while by removing the first element we obtain a set which does not satisfy condition (iii) above.

We denote by  $vol_{\mathcal C} (k,T)$ the maximum volume of a chain with cardinality $k$ and doubling $T$. We say that a chain $A$ is {\it extremal} if $vol (A)=vol_{\mathcal C} (|A|,|2A|)$. 

Arithmetic progressions are chains. The examples of extremal sets given after Theorem \ref{thm:3k-4} are chains (by removing its smaller element we obtain an arithmetic progression.) We  note that, if $A$ is a chain, then $D(A)$ is also a chain. It follows from the above remarks and Corollary \ref{cor:ub} that,  for each $k$ and each $T\in [2k,{k\choose 2}+1]$, we have
$$
vol_{\mathcal C} (k,T)\ge \mu (k,T)+1.
$$
Theorem \ref{thm:chain} states that equality holds. The remainder of this Section contains a sequence of Lemmas which provide the proof of Theorem \ref{thm:chain}. Some of the statements apply to the larger class of extremal one--dimensional sets. We say that a set $A$ is $1$--{\it extremal} if it has largest volume among all one--dimensional sets with its same cardinality and doubling. 

We collect some observations in the following Lemma for future reference.
  
  \begin{lemma}\label{lem:cx} Let $A$ and $A_x=A\cup \{ x\}$,  $x>\max (A)$, be two one--dimensional sets. Set $k=|A|, T=|2A|, T_x=|2A_x|$ and $\Delta T=T_x-T$. The following holds:
 \begin{equation}\label{eq:deltaT0}
 \Delta T=k+1-|2A\cap (A+x)|,
 \end{equation}
 and
  \begin{equation}\label{eq:deltaT1}
2\le  \Delta T\le k,
 \end{equation}
 and
 \begin{equation}\label{eq:cx0}
 c(k,T)-1\le c(k+1,T_x)\le c(k,T)+1.
 \end{equation}
 \end{lemma}
 
 \begin{proof} We have
 $$
 T_x=|2A\cup (x+A)\cup \{2x\}|= |2A|+|A|-|2A\cap (x+A)|+1.
 $$
 Since $A_x$ is one--dimensional, $x$ must be involved in an additive relation with the elements in $A$, which implies  $|2A\cap (x+A)|\ge 1$. On the other hand, $x>\max (A)$ implies $|2A\cap (x+A)|\le |A|-1$. This gives  \eqref{eq:deltaT0} and \eqref{eq:deltaT1}.
 
 For the last part, write $c=c(k,T)$ and just note that, for $c=2$, $T\in [2k-1,3k-4]$ and $T+[2,k]\subset [2(k+1)-1,4(k+1)-8]=I_{2,k+1}\cup I_{3,k+1}$, while, for $c\ge 3$, 
 $$
 T\in I_{c,k}=\left[ ck-{c+1\choose 2}+3, ck-{c+1\choose 2}+(k-c)+1\right],
 $$ 
 and the interval  
 $$
I_{c-1,k+1}\cup  I_{c,k+1}\cup I_{c+1,k+1}=\left[(c-1)(k+1)-{c\choose 2}+3, (c+1)(k+1)-{c+2\choose 2}+(k-c)+1\right],
 $$
 contains the values of $T+\Delta T$ for each $\Delta T\in [2,k]$.
 \end{proof}

Our next Lemma gives  a lower bound on the values of integers $x$ which can be added to a set $A$ which admits a stable decomposition,  if one expects to obtain an extremal set. 

\begin{lemma}\label{lem:int0} Let $A=A_1\circ P \circ A_2$ be a  set with $k=|A|, T=|2A|$ and $a=\max (A)$, 
where $A_1$ is an stable set and $A_2$ is  a right stable set and $P$ is a segment. 

Let $A_x=A\cup \{x\}$ with $a<x\le 2a$ and $T_x=|2A_x|$. 

Assume that $T_x>3(k+1)-4$ and that $A_x$ is $1$--dimensional.

If $a= \mu (k,T)$ and $x\ge \mu (k+1, T_x)$ then $$x\ge 2a-(a_1+a_2-2),$$
where $a_1=\max(A_1)$ and $a_2=\max(A_2)$.
\end{lemma}

\begin{proof} Since $A_1$ is stable and $A_2$ is right stable, we have
$$
2A=A_1\circ Y\circ A_2,
$$
for some set $Y$. From \eqref{eq:deltaT0},
\begin{equation}\label{eq:tx0}
\Delta T=T_x-T=|A|+1-|2A\cap (x+A)|.
\end{equation}
The intersection $2A\cap (x+A)$ is contained in the interval $[x,2a]$. If $x<2a-(a_1+a_2-2)$ then all holes of $A$ are also holes in $2A\cap (x+A)$ (see Figure \ref{fig:int0} for an illustration). It follows that
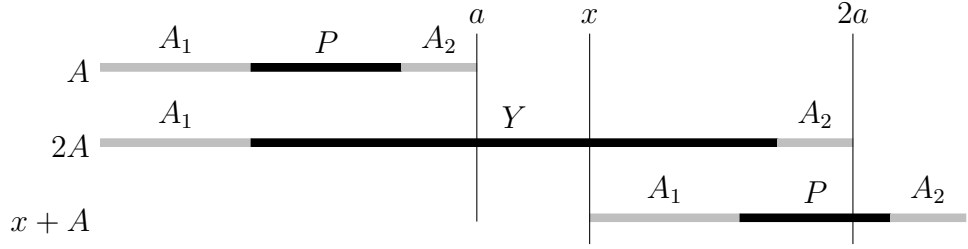
\begin{figure}
		\hspace{-3.5cm}	\begin{tikzpicture}
		\draw[lightgray,fill] (0,2) rectangle (2,2.1);
		\draw[fill] (2,2) rectangle (4,2.1);
		\draw[lightgray,fill] (4,2) rectangle (5,2.1);
		\draw[lightgray,fill] (0,1) rectangle (2,1.1);
		\draw[fill] (2,1) rectangle (9,1.1);
		\draw[lightgray,fill] (9,1) rectangle (10,1.1);
		\node[left] at (0,2) {$A$};
		\node[left] at (0,1) {$2A$};
		\node[left] at (0,0) {$x+A$};
		\node[above] at (1,2.1) {$A_1$};
		\node[above] at (3,2.1) {$P$};
		\node[above] at (4.5,2.1) {$A_2$};
		\node[above] at (1,1.1) {$A_1$};
		\node[above] at (5.5,1.1) {$Y$};
		\node[above] at (9.5,1.1) {$A_2$};
		\draw (5,2.5)--(5,0);
		\node[above] at (5,2.5) {$a$};
		\draw (10,2.5)--(10,-0.3);
		\node[above] at (10,2.5) {$2a$};
		\draw[lightgray,fill] (6.5,0) rectangle (8.5,0.1);
		\draw[fill] (8.5,0) rectangle (10.5,0.1);
		\draw[lightgray,fill] (10.5,0) rectangle (11.5,0.1);
		\draw (6.5,2.5)--(6.5,-0.3);
		\node[above] at (6.5,2.5) {$x$};
		\node[above] at (7.5,0.1) {$A_1$};
		\node[above] at (9.5,0.1) {$P$};
		\node[above] at (11,0.1) {$A_2$};
		
		\end{tikzpicture}		
		\caption{Illustration of the evaluation of $|2A\cap (x+A)|$.}\label{fig:int0}
	\end{figure}
\begin{equation}\label{eq:int0}
|2A\cap (x+A)|\le 2a-x+1-(a-|A|+1)=a-x+|A|.
\end{equation}
Let  $\Delta\mu=\mu (k+1,T_x)-\mu (k,T)$. By combining \eqref{eq:tx0} and \eqref{eq:int0} and using the hypothesis on $x$ and $a$  we obtain
\begin{equation}\label{eq:tmu}
\Delta T\ge x-a+1\ge \Delta \mu+1.
\end{equation}
Let $c=c(k,T), b=b(k,T)$ and $c_x=c(k+1,T_x), b_x=b(k+1,T_x)$. It follows from \eqref{eq:cx0} that  $c-1\le c_x\le c+1$. 

 If  $c_x=c$ then we have, by using \eqref{eq:mudef} and \eqref{eq:kbdef},  
 $$
 \Delta T=c+b_x-b\; \text{ and } \; \Delta \mu = 2^{c-2}(b_x-b+1).
 $$ 
 Since $T_x>3(k+1)-4$ we have $c_x\ge 3$. Moreover,  $c_x=c$ implies $b_x\ge b$. Since $c\le 2^{c-2}+1$ for all $c\ge 3$ it follows that \eqref{eq:tmu} does not hold, a contradiction.
 
 Suppose now  $c_x=c+1$. In this case,   
 $$\Delta T=k+b_x-b\; \text{ and } \; \Delta \mu= 2^{c-2}(k-c+2b_x-b+1).$$
 If $c=2$ then $\Delta \mu+1=\Delta T+b_x>\Delta T$ and \eqref{eq:tmu} does not hold. Suppose $c\ge 3$. 
 Since $b\le k-c-1$, we have $k-c+2b_x-b+1\ge 2b_x+2\ge 4$. Hence, 
 $$
 \Delta \mu +1\ge (2^{c-2}-1)4+\Delta T-c+2\ge \Delta T+2^{c}-c-2>\Delta T,
 $$
 which again contradicts \eqref{eq:tmu}.  
 
 Finally, if $c_x=c-1$ then  $c\ge 3$ and 
 $$
 \Delta T=2c+b_x-b-(k+1)\; \text{ and}\;\;  \Delta \mu=2^{c-3}(c+b_x-2b-(k+1)).
 $$
Hence we can write  $\Delta \mu=2^{c-3}(\Delta T-c-b)$. It follows from the proof of Lemma \ref{lem:cx} that $c_x=c-1$ implies $\Delta T<c$ so that we get $\Delta \mu<0$, a contradiction.
 This completes the proof.
 \end{proof}

\subsection{Crossing  $3k-4$}	
The next Lemma gives tight conditions on an  extremal set with doubling smaller than $3k-4$ to be extendable to a $1$--extremal set with doubling larger than $3(k+1)-4$.

\begin{lemma}\label{lem:chain} Let $A$ be a  extremal set with  $k=|A|$, $a=\max (A)$ and
$$
T=|2A|=2k-1+b,\; 1\le b\le k-3.
$$
Let $A_1\circ P\circ A_2$ be the stable decomposition of $A$ with $a_1=\max (A_1)$ and $a_2=\max (A_2)$.

If   $A_x=A\cup \{x\}$, $a< x\le 2a$,  is $1$--extremal with $T_x=|2A_x|\ge 3(k+1)-3$  then 
\begin{equation}\label{eq:xismu}
x=\mu (k+1, T_x).
\end{equation}
Moreover, 
\begin{equation}\label{eq:a1a2dense}
|A\cap (x-a+A)|=\left\lceil \frac{2a-x+1}{2}\right\rceil.
\end{equation}
\end{lemma}

\begin{proof} Assume that $A_x$ is $1$--extremal. Let  $T=|2A|$ and $T_x=|2A_x|$.  By Theorem \ref{thm:3k-4} we have 
		\begin{equation}\label{eq:a}
	a=k-1+b,
	\end{equation}
	and
	\begin{equation}\label{eq:2all}
	2A=A_1\circ P'\circ A_2,
	\end{equation}
	where
	$P'$ is a progression with length
	\begin{equation}\label{eq:p'}
	|P'|=2a+1-(a_1+a_2).
	\end{equation}
Moreover, since $T_x>3(k+1)-4$, we have  $x> 2(k+1)-4$.
We write the integers in the interval $[2(k+1)-3,4(k+1)-3]$ as
	\begin{equation}\label{eq:x}
	x=2(k+1)-4+2b_x+\delta, \;\;  1\le b_x\le (k+1)-4,\; 0\le \delta \le1.
	\end{equation}
	
	By Corollary \ref{cor:ub}, since  $A_x$ is $1$--extremal, we have
	\begin{equation}\label{eq:tx}
	T_x\le 3(k+1)-4+b_x=T+k+b_x-b.
	\end{equation}
	Indeed,  the above inequality holds when $\delta =0$  in \eqref{eq:x} by Corollary \ref{cor:ub}. When $\delta =1$ then the inequality also holds because $\mu (k+1,T_x+1)=\mu(k+1,T_x)+2$ for our  range of $T_x$.

	We have
	\begin{equation}\label{eq:union}
	\Delta T=T_x-T=(k+1)-|2A\cap (x+A)|.
	\end{equation}
	
	It follows from  inequality \eqref{eq:tx} and \eqref{eq:union} that
	\begin{equation}\label{eq:int}
	|2A\cap (x+A)|\ge b-b_x+1.
	\end{equation}
	We note that $2A\cap (x+A)$ is contained in the interval $[x,2a]$ whose length, by \eqref{eq:a} and \eqref{eq:x}, is
	\begin{equation}\label{eq:2x-a}
	2a-x+1=2(b-b_x)+1-\delta.
	\end{equation}

	By Lemma \ref{lem:int0} we have $x\ge 2a-(a_1+a_2-2)$. Therefore,  $2A\cap (x+A)$ contains at most elements coming from  the initial segment of the stable set $A_1$ or from the final segment of the right stable set $A_2$ (see Figure \ref{fig:int} for an illustration). Therefore, Lemma  \ref{lem:hole} gives

	\begin{figure}		
		\begin{tikzpicture}
		\draw[lightgray,fill] (0,2) rectangle (2,2.1);
		\draw[fill] (2,2) rectangle (4,2.1);
		\draw[lightgray,fill] (4,2) rectangle (5,2.1);
		\draw[lightgray,fill] (0,1) rectangle (2,1.1);
		\draw[fill] (2,1) rectangle (9,1.1);
		\draw[lightgray,fill] (9,1) rectangle (10,1.1);
		\node[left] at (0,2) {$A$};
		\node[left] at (0,1) {$2A$};
		\node[left] at (0,0) {$x+A$};
		\node[above] at (1,2.1) {$A_1$};
		\node[above] at (3,2.1) {$P$};
		\node[above] at (4.5,2.1) {$A_2$};
		\node[above] at (1,1.1) {$A_1$};
		\node[above] at (5.5,1.1) {$P'$};
		\node[above] at (9.5,1.1) {$A_2$};
		\draw (5,2.5)--(5,0);
		\node[above] at (5,2.5) {$a$};
		\draw (10,2.5)--(10,-0.3);
		\node[above] at (10,2.5) {$2a$};
		\draw[lightgray,fill] (7.5,0) rectangle (9.5,0.1);
		\draw[fill] (9.5,0) rectangle (13.5,0.1);
		\draw[lightgray,fill] (13.5,0) rectangle (14.5,0.1);
		\draw (7.5,2.5)--(7.5,-0.3);
		\node[above] at (7.5,2.5) {$x$};
		\node[above] at (8.5,0.1) {$A_1$};
		\node[above] at (11.5,0.1) {$P$};
		\node[above] at (14,0.1) {$A_2$};
		
		\end{tikzpicture}
		
		\caption{Illustration of the evaluation of $|2A\cap (x+A)|$.}\label{fig:int}
	\end{figure}
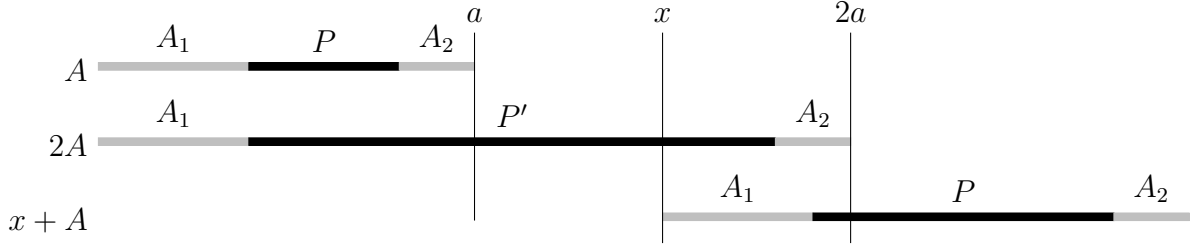

	\begin{equation}\label{eq:b-bx}
	|2A\cap (x+A)|\le  \left\lceil \frac{2a-x+1}{2}\right\rceil =b-b_x+1-\delta.
	\end{equation}
	Hence, according to  \eqref{eq:int}, the only possibility for $A_x$ to be extremal is that 
	$$
	\delta =0,
	$$ 
	and  equality holds in \eqref{eq:b-bx}, which can be rewritten as \eqref{eq:a1a2dense} in the Lemma; indeed, by \eqref{eq:p'}, we have $2A\cap [x,2a]=[x,2a-a_2]\circ A_2$ and hence $2A\cap (x+A)=(a+A)\cap (x+A)$.  Moreover, equality  in \eqref{eq:b-bx} implies that there is also equality in \eqref{eq:tx} which implies that   $x$ is given precisely by $\mu(k+1,T_x)$, proving \eqref{eq:xismu}. \end{proof}
	 
In the case that the extremal set $A$ in Lemma \ref{lem:chain} is a chain then the next lemmas further show that $D(A)$ is the only possible extension of  $A$ to a larger chain unless both stable sets in the stable decomposition of $A$ are $2$--progressions. 

\begin{lemma}\label{lem:consec} Let $A$ be a normal set of integers with $k=|A|$ and $|2A|\le 3k-4$ with stable decomposition $A=A_1\circ P\circ A_2$. If $A$ is a chain then $A_1$ and $A_2$ contain no pair of consecutive elements.
\end{lemma}

\begin{proof}  Suppose that $A_1$ or $A_2$ contain a pair of consecutive elements. By replacing $A$ by its reflexion $A^-$ if necessary, we may assume that $A_1$ contains two consecutive elements. Let $a_1=\max (A_1)$. Let $y$ be the largest element in $A_1$ such that $y+1\in A_1$ and $y+2\not\in A_1$ (such element exists since $a_1-1\not\in A_1$.) Since $|P|\ge 3$, we have $a_1-1\not\in A$, and $\{a_1,a_1+1\}\subset A$.  

Let $A'$ be the minimal subchain of $A$ containing $\{y,a_1+1\}$. Then  $|2A'|\le 3|A'|-4$ and  $\min (A')=y$ or $\max (A')=a_1+1$. Assume $\min (A')=y$, the other case being similar. By Theorem \ref{thm:3k-4},  $A'$ can not be extremal since, once normalized, it has the form $\{0,1,3,\ldots\}$ and hence, it has no stable decomposition. This contradicts that $A$ is a chain.  
\end{proof}

\begin{lemma}\label{cor:doubling} Let $A$ and $A_x$  be as in  Lemma \ref{lem:chain}. 
	
	Assume that $A$ is a chain. 
	
Then   $A_x$ is $1$--extremal for some $x<2a$ if and only if  $x$ is even, $2a-x\ge a_1+a_2-2$, both $A_1$ and $A_2$ are $2$--progressions, and $(a-x+A)\cap A$ is a $2$--progression. In this case, $$A_x=A_1\circ P\circ A'_2,$$ where $A'_2= A_2\cup \{x-(a-a_2)\}$ is right--stable.
\end{lemma}

\begin{proof} Suppose that $A_x$ is extremal for some $x<2a$.  
	
By  Lemma \ref{lem:chain}  we have $2a-x\le a_1+a_2-2$ and  $x=\mu (k+1,T_x)$, which is an even number, and $X=A\cap (a-x+A) \subset [0,2a-x]$ contains $(2a-x)/2+1$ elements.

Since $2a-x$ is even either $X$ is a $2$--progression or it contains two consecutive elements. Suppose the latter holds. As the length of $X$ is at most $2a-x\le a_1+a_2-2$, the set $X$ consists of an initial segment of $A_1$ and a final segment of $A_2$ (see Figure \ref{fig:int} for an illustration.) Hence $A_1$ or $A_2$ contain two consecutive elements. By Lemma \ref{lem:consec}  this contradicts that $A$ is a chain. 

Hence $X$ is a $2$--progression, which implies that  each of $A_1$ and $A_2$ must be a $2$--progression.  This completes the proof of the `if' part of the statement.

Suppose now that each of $A_1$, $A_2$ and $X$ is a $2$--progression, $x$ is even and $2a-x\ge a_1+a_2-1$. Then there is equality in \eqref{eq:b-bx} and \eqref{eq:int}, which implies that $A_x$ is $1$--extremal. 

We note that $A_x=(A_1\circ P\circ A_2)\cup \{x\}=A_1\circ P\circ (A_2\cup \{x-(a-a_2)\})$. Since each of $A_2$ and $X=(a-x+A)\cap A$ is a $2$--progression, then
$(A'_2)^-=\{ 0\}\cup ((x-a)+A_2^-)$ is also stable. This completes the proof.
 \end{proof}

\subsection{Beyond $3k-4$} 

The next step towards the proof of Theorem \ref{thm:chain} is to show that, under mild conditions, a $1$--extremal set of the form $D(A)$ can be only extended to a larger chain by iterating the operator $D$.

\begin{lemma}\label{lem:unique} Let $A$ be a $1$--extremal  set  with $k=|A|$, $T=|2A|$ and $a=\max (A)$, and let   $B=D(A)$.
	
	If  $B_x=B\cup \{x\}$ is $1$--extremal for some $2a<x\le 4a$ then $x\in \{3a,4a\}$.
	
	Moreover, if  $a=\mu(k,T)$ and  $\mu (k,T)>2^c$, where $c=c(k,T)$, then
	$$B_x=D^2(A ).$$ 
\end{lemma}

\begin{proof} Suppose  that $B_x=B\cup \{ x\}$ is extremal for some $2a<x\le 4a$. As usual we consider
	\begin{equation}\label{eq:2bx}
		|2B_x|=|2B|+|B|+1-|(x+B)\cap 2B|.
	\end{equation}
	Since $B=D(A)=A\cup \{2a\}$, we have
	\begin{equation}
		|(x+B)\cap 2B|=|(x+A)\cap (2a+A)|\label{eq:xB}
	\end{equation}
	If $x>3a$ then  $|(x+B)\cap 2B|\le 1$ (the intersection  contains at most $4a$) and, according to \eqref{eq:2bx},  $B_x$ has not smaller doubling than $D(B)$ but has smaller volume unless $x=4a$. On the other hand, if $2a<x<3a$ then  $B'=A\cup \{x-a\}$ has doubling (see Figure \ref{fig:bx} for an illustration)
	\begin{align*}
		|2B'|&=|2A|+|A|+1-|(x-a+A)\cap 2A)|\\
		&\le |2A|+|A|+1-|(x-a+A)\cap (a+A)|\\
		&\stackrel{\eqref{eq:xB}}{=}|2B|+1-|(x+B)\cap 2B|\\
		&\stackrel{\eqref{eq:2bx}}{=}|2B_x|-|B|.
	\end{align*}
	It follows that $D(B')$ has the same cardinality as $B_x$ and $|2D(B')|=|2B'|+|B|\le |2B_x|$, while its larger element is  $2(x-a)=x+(x-2a)>x$, contradicting that $B_x$ is extremal. 
	
	\begin{figure}[h]
		\begin{center}
			\begin{tikzpicture}
			\foreach \i in {0,1,2,3,4}
			{
				
				\draw[dotted] (2*\i,-0.5)--(2*\i,4.5);
				\draw (-0.5,\i)--(10,\i);
			}
			
			\node[above] at (0,4.5) {$0$};
			\node[above] at (2,4.5) {$a$};
			\node[above] at (4,4.5) {$2a$};
			\node[above] at (6,4.5) {$3a$};
			\node[above] at (8,4.5) {$4a$};
			
			\draw[lightgray,fill] (0,3.9) rectangle (2,4.1);
			\draw[fill] (4,4) circle (2pt);
			\node[left] at (-0.5,4) {$B$};
			\node[above] at (1,4) {$A$};
			
			\draw[gray,fill] (0,2.9) rectangle (4,3.1);
			\draw[lightgray,fill] (4,2.9) rectangle (6,3.1);
			\draw[fill] (8,3) circle (2pt);
			\node[left] at (-0.5,3) {$2B$};
			\node[above] at (2,3) {$2A$};
			\node[above] at (5,3) {$A$};

			\draw[lightgray,fill] (5.2,1.9) rectangle (7.2,2.1);
			\draw[fill] (9.2,2) circle (2pt);
			\node[left] at (-0.5,2) {$x+B$};
			\node[above] at (6,2) {$x+A$};
			\draw[dotted] (5.2,2)--(5.2,4.5);
			\node[above] at (5.2,4.5) {$x$};
			
			\draw[lightgray,fill] (0,0.9) rectangle (2,1.1);
			\draw[fill] (3.2,1) circle (2pt);
			\node[left] at (-0.5,1) {$B'$};
			\node[above] at (1,1) {$A$};

			\draw[lightgray,fill] (3.2,-0.1) rectangle (5.2,0.1);
			\draw[gray,fill] (0,-0.1) rectangle (4,0.1);
			\draw[fill] (6.4,0) circle (2pt);
			\node[left] at (-0.5,0) {$2B'$};
			
			\end{tikzpicture}
		\end{center}
		\caption{An illustration of the computation of $2B_x$ and $2B'$.}\label{fig:bx}
	\end{figure}
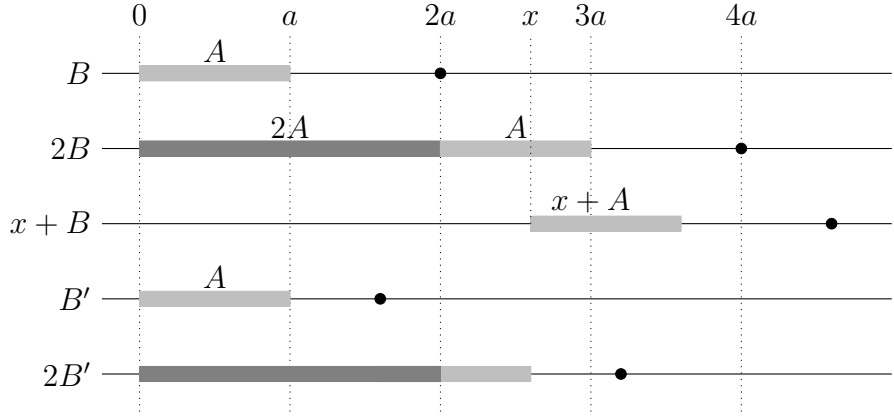

	If $x=3a$,  then   $|(x+B)\cap 2B|=2$ which yields $$|T_x|=|2B_x|=|2B|+|B|-1=|T|+2k.$$ 
	
	Assume now that $a=\mu (k,T)$ and $\mu (k,T)>2^c$. We have $2a=\mu (k+1,T+k)$ and 
	$$4a=\mu (k+2,T+2k+1)=\mu (k+2, T_x+1)=\mu (k+2, T_x)+2^{c'-2},$$
	where $c'=c(k+2,T_x)\le c+2$. If $B_x$ with $x=3a$ is extremal, we have 
	$$3a\ge \mu (k+2, T_x)\ge 4a-2^c,$$ and hence $a\le 2^c$, a contradiction.  It follows that $B_x$ is not extremal for $x=3a$ and the only choice left is $x=4a$. Thus $B_x=D^2(A)$.
\end{proof}

We remark that the condition $\mu (k,T)=2^{c-2}(k-c+b+1)>2^c$, $c=c(k,T)$, in Lemma \ref{lem:unique} is only violated when $k\le c-b+3$, namely when $b=1$ and $c=k-2$. In such cases, indeed, the choice $x=3a$ in Lemma \ref{lem:unique} can give rise to an extremal set which is not of the form $D(B)$ for some $B$. A simple example is $B=\{0,1,2,4\}=D(\{0,1,2\})$ which is contained in the $1$--extremal set $\{0,1,2,4,8\}\neq D(B)$.  By Lemma \ref{lem:cx}, if the doubling constant of a chain $A$ is smaller than $k-2$ then any chain $B$ containing $A$ has also doubling constant smaller than $|B|-2$. Therefore all examples of chains for which the condition $\mu (k,T)=2^{c-2}(k-c+b+1)>2^c$, $c=c(k,T)$, in Lemma \ref{lem:unique} does not apply are the ones obtained from $\{0,1,2,4\}$. We shall discuss the structure of chains arising from this kind of examples later on.

Lemma \ref{lem:unique} only handles the case when a $1$--extremal set of the form $B=D(A)$  is extended by adding one element to the right. Next Lemma considers extending $B$ to the left, namely, extending its reflexion $B^-$ by adding one element to the right.  As it happens, showing that again $D(B^-)$ is the only extension of $B^-$ to a larger chain, requires the use of the full strength of the assumptions, namely, that $A\cup \{y\}$ is not a chain for $\max(A)<y<2\max(A)$. 

\begin{lemma}\label{lem:lunique} Let $A$ be a chain with $k=|A|, T=|2A|$ and    $a=\max (A)$. Assume that $a=\mu (k,T)$ and that  $\mu (k,T)>2^c,\; c=c(k,T)$, and 
	$$
	y< \min\{ \mu (k+1,|2(A\cup \{y\})|),  \mu (k+1,|2(A^-\cup \{y\})|)\},\; \text{for all } a<y<2a.
	$$
	
	Let $B=(D(A))^-$.   
	
	Then $B_x=B\cup \{x\}$ is extremal for some $2a<x\le 4a$ if and only if $$B_x=D(B).$$	
\end{lemma} 

\begin{proof} Suppose  that $B_x=B\cup \{ x\}$ is extremal for some $2a<x\le 4a$. We have
	$$
	|2B_x|=|2B|+|B|+1-|(x+B)\cap 2B|.
	$$
	We again consider three cases (see an illustration in Figure \ref{fig:bx-}). 
	
	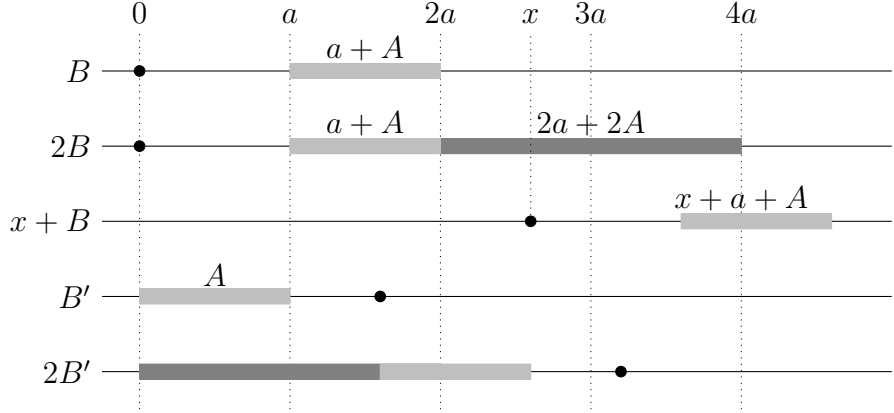
\begin{figure}[h]
		\begin{center}
			\begin{tikzpicture}
			\foreach \i in {0,1,2,3,4}
			{
				
				\draw[dotted] (2*\i,-0.5)--(2*\i,4.5);
				\draw (-0.5,\i)--(10,\i);
			}
			
			\node[above] at (0,4.5) {$0$};
			\node[above] at (2,4.5) {$a$};
			\node[above] at (4,4.5) {$2a$};
			\node[above] at (6,4.5) {$3a$};
			\node[above] at (8,4.5) {$4a$};
			
			\draw[lightgray,fill] (2,3.9) rectangle (4,4.1);
			\draw[fill] (0,4) circle (2pt);
			\node[left] at (-0.5,4) {$B$};
			\node[above] at (3,4) {$a+A$};
			
			\draw[gray,fill] (4,2.9) rectangle (8,3.1);
			\draw[lightgray,fill] (2,2.9) rectangle (4,3.1);
			\draw[fill] (0,3) circle (2pt);
			\node[left] at (-0.5,3) {$2B$};
			\node[above] at (6,3) {$2a+2A$};
			\node[above] at (3,3) {$a+A$};

			\draw[lightgray,fill] (7.2,1.9) rectangle (9.2,2.1);
			\draw[fill] (5.2,2) circle (2pt);
			\node[left] at (-0.5,2) {$x+B$};
			\node[above] at (8,2) {$x+a+A$};
			\draw[dotted] (5.2,2)--(5.2,4.5);
			\node[above] at (5.2,4.5) {$x$};
			
			\draw[lightgray,fill] (0,0.9) rectangle (2,1.1);
			\draw[fill] (3.2,1) circle (2pt);
			\node[left] at (-0.5,1) {$B'$};
			\node[above] at (1,1) {$A$};

			\draw[gray,fill] (0,-0.1) rectangle (4,0.1);
			\draw[lightgray,fill] (3.2,-0.1) rectangle (5.2,0.1);
			
			\draw[fill] (6.4,0) circle (2pt);
			\node[left] at (-0.5,0) {$2B'$};
			
			\end{tikzpicture}
		\end{center}
		\caption{An illustration of the computation of $2B_x$ and $2B'$.}\label{fig:bx-}
	\end{figure}	
	
	Suppose first that $3a<x\le 4a$. Then $(x+B)\cap 2B$ contains at most the point $x$. Hence,  $B_x$ has no smaller doubling than $D(B)$ but has a  smaller volume unless $x=4a$. 
	
	Consider now $2a<x<3a$. Set $k=|A|, T=|2A|, c=c(k,T)$ and $b=b(k,T)$. Since $B_x$ is $1$--dimensional we have $x\in 2(\{0\}\cup (a+A))$, which implies $x\in 2a+2A$. Hence,
	\begin{align*}
		T_x=|2B_x|&=|2(\{0\}\cup (a+A)\cup \{x\})|\\
		&= 1+|(a+A^-)\cup (x+a+A) \cup (2a+2A)|+1\\
		&=1+|a+A|+|(x+a+A) \cup (2a+2A)|.
	\end{align*}
	On the other hand, for $B'=A\cup \{x-a\}$ we have
	$$
	T'=|2B'|=|(x-a+A)\cup 2A|+1,
	$$
	and
	$$
	|2D(B')|=|B'|+|2B'|=1+|A|+|(x+a+A)\cup (2a+2A)|+1=T_x+1.
	$$
	By assumption,  $x-a<\mu(k+1,T')$ and hence,
	\begin{align}
		2(x-a)&<\mu (k+2,T'+k+1)\nonumber\\
		&=\mu (k+2,T_x+1)\nonumber\\
		&=\mu (k+2,T_x)+2^{c_x-2}\nonumber\\
		&\le x+2^{c_x-2},\label{eq:x2a} 
	\end{align}	
	where $c_x=c(k+2,T_x)$ satisfies $$c+1\le c_x\le c+2.$$ 
	 It follows that
	$$
	\mu (k+2,T_x)\le x<2a+2^{c_x-2},
	$$
	the lower bound since $B_x$ is extremal, the upper bound from \eqref{eq:x2a}.
	By plugging in the values of $\mu (k+2,T_x)=2^{c_x-2}(k+2-c_x+b_x+1)$ and $a=\mu (k,T)=2^{c-2}(k-c+b+1)$ we obtain
	$$
	2^{c_x-2}(k+2-c_x+b_x+1)<2^{c-1}(k-c+b+1).
	$$
	If $c_x=c+2$ the above inequality leads to $b_x<b-(k+1)<0$, which is not possible. If $c_x=c+1$ then we get $
	b_x<b-1$, and hence $T_x<T+k$, again a  contradiction.	
	
	Finally, if $x=3a$,  then   $|(x+B)\cap 2B|=2$ which yields $$|T_x|=|2B_x|=|2B|+|B|-1=|T|+2k.$$ 
	
	Assume now that $\mu (k,T)>2^c$. We have $2a=\mu (k+1,T+k)$ and 
	$$4a=\mu (k+2,T+2k+1)=\mu (k+2, T_x+1)=\mu (k+2, T_x)+2^{c_x-2}.$$
	If $B_x$ with $x=3a$ is $1$--extremal, we have 
	$$3a\ge \mu (k+2, T_x)\ge 4a-2^c,$$ and hence $a\le 2^c$, a contradiction.  It follows that $B_x$ is not $1$--extremal for $x=3a$ and the only choice left is $x=4a$. Thus $B_x=D^2(A)$.		
\end{proof}

\subsection{The case of $2$--progressions}

It follows from Lemma  \ref{cor:doubling} that a chain $A$ with doubling $T\le 3|A|-4$ can only be extended to $D(A)$ unless both stable sets in its stable decomposition are $2$--progressions. If this is not the case then lemmas \ref{lem:unique} and \ref{lem:lunique} show that every further extension is obtained by iterate applications of the operator $D$, which proves Theorem \ref{thm:chain} in this situation. In order to complete the proof of Theorem \ref{thm:chain} it remains to analyze the case of stable sets which are $2$--progressions. This is the purpose of the final part of this Section.

\begin{lemma}\label{lem:2prog2odd} Let $A=A_1\circ P\circ A_2$ and $k=|A|$. Assume that both of $A_1$ and $A_2$ are $2$--progressions and $|P|\ge 4$. 

Let $y>x>a=\max (A)$ such that $T_x=|2A_x|>3(k+1)-4$, $A_x=A\cup \{x\}$.
 
\begin{enumerate}
\item[(i)] If $A_{xy}=A_x\cup \{y\}$ is  a chain then $A_{xy}=D(A_x).$
\item[(ii)] If $A_{xy}'=(A_x)^-\cup \{y\}$ is a chain and $y>x+2$ then $A_{xy}=D((A_x)^-)$.
\end{enumerate}
\end{lemma}

\begin{proof} By the structure of $A$ we have $|2A|\le 3k-4$ and $A$ is extremal. By Lemma \ref{lem:chain} we have 
$x=\mu (k+1,T_x)\ge  2a-(a_1+a_2-2)=a+|P|$, an even number. Hence, since $|P|\ge 4$ and $a$ is odd when $|P|=4$, we have 
\begin{equation}\label{eq:x+a}
x>a+4.
\end{equation}
Let 
$$
T=|2D(A_x)|=|2D((A_x)^-)|=T_x+k+1.
$$
\begin{enumerate}
\item[(i)] Let $g(y)=|(y+A_x)\cap 2A_x|$, so that
$$T_{xy}=|2A_{xy}|=T_x+(k+2)-g(y)=T+1-g(y).$$ 
We have $A_x=A_1\circ P \circ A'_2$ where $A'_2$ is right stable, and therefore $2A_x=A_1\circ Y \circ A'_2$ for some $Y$. 
By Lemma \ref{lem:int0} applied to $A_x$ and $A_{xy}$ we have $y\ge x+a-(a_1+a_2-2)$. Moreover, since each of $A_1, A_2$ is a  $2$--progression, by the structure of $2A_x$ we have (see Figure \ref{fig:2prog} for an illustration)
\begin{equation}\label{eq:g(y)}
g(y)\le \left\{\begin{array}{ll} 1, &  x+a< y\le  2x\\ \lfloor h/2\rfloor +2, & y=x+a-h,\; 0\le h\le a_1+a_2-2 \end{array}\right. .
\end{equation}
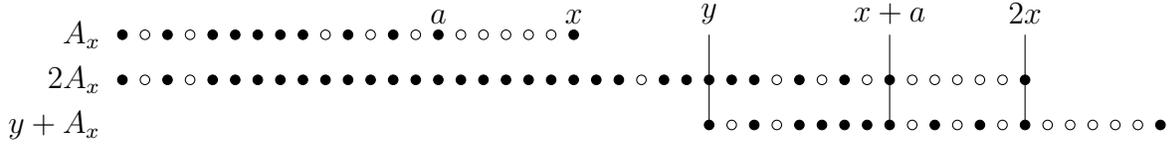
\begin{figure}[h]
		\begin{center}
			\begin{tikzpicture}[scale=0.3]
			\foreach \i in {0,2,4,5,6,7,8,10,12,14,20}
			{
			\draw[fill] (\i,4) circle (6pt);	
			}
			\foreach \i in {1,3,9,11,13,15,16,17,18,19}
			{
			\draw (\i,4) circle (6pt);	
			}
			\node[left] at (-0.5,4) {$A_x$};
			\foreach \i in {4,...,22}
			{
			\draw[fill] (\i,2) circle (6pt);	
			}
			\foreach \i in {0,2,24, 25, 26, 27, 28, 30, 32, 34, 40}
			{
			\draw[fill] (\i,2) circle (6pt);	
			}
			\foreach \i in {1,3,23,29,31,33,35,36,37,38,39}
			{
			\draw (\i,2) circle (6pt);	
			}
			\node[left] at (-0.5,2) {$2A_x$};

\foreach \i in {0,2,4,5,6,7,8,10,12,14,20}
			{
			\draw[fill] (\i+26,0) circle (6pt);	
			}
			\foreach \i in {1,3,9,11,13,15,16,17,18,19}
			{
			\draw (\i+26,0) circle (6pt);	
			}
			\node[left] at (-0.5,0) {$y+A_x$};
			\draw (26,4)--(26,0);
			\draw (40,4)--(40,0);
			\draw (34,4)--(34,0);
			\node[above] at (26,4) {$y$};
			\node[above] at (40,4) {$2x$};
			\node[above] at (34,4) {$x+a$};
			\node[above] at (14,4) {$a$};
			\node[above] at (20,4) {$x$};
			
			\end{tikzpicture}
		\end{center}
		\caption{An illustration of the computation of $g(y)=|(y+A_x)\cap 2A_x|$.}\label{fig:2prog}
	\end{figure}

If $y=2x$ then $A_{xy}$ is $1$--extremal with $T_{xy}=T$ and  $g(y)=1$.

According to \eqref{eq:g(y)}, the largest value of $y$ for which $T_{xy}=T-1$, namely $g(y)=2$,  is $y=x+a$. Since $c(k+2,T)=4$, we have $\mu (k+2,T-1)=\mu (k+2,T)-4=2x-4$. By \eqref{eq:x+a} we have $x+a<2x-4$. Hence    $A_{xy}$ is not $1$--extremal for $y=x+a$. Moreover, we have $c(k+2,T-h')=3$ and  $\mu (k+2,T-h'-1)=\mu (k+2,T-h')-2$ for each $h'$ such that $T_x<T-h'-1< T-1$. According to \eqref{eq:g(y)}, the largest $y$ for which $T_{xy}=T-h'-1$, namely $g(y)=h'+2$, is 
$$y=x+a-2h'<\mu (k+2,T-1)-2h'=\mu (k+2,T-h'-1).$$
Hence $A_{xy}$ is not $1$--extremal for all the remaining values of $y$. Hence, if $A_{xy}$ is $1$--extremal we must have $y=2x$ and $A_{xy}=D(A_x)$.
 This completes the proof of (i).
 
\item[(ii)] The proof follows the same lines as the above one. Let $g(y)=|(y+(A_x)^-)\cap 2(A_x)^-|$, so that
$$T_{xy}'=|2A_{xy}'|=T_x+(k+2)-g(y)=T+1-g(y).$$ 
We have $(A_x)^-=(A'_2)^-\circ P \circ A_1$ where $A_1$ is right stable, and therefore $2(A_x)^-=(A'_2)^-\circ Y \circ A_1$ for some $Y$. 
By Lemma \ref{lem:int0} applied to $(A_x)^-$ and $A_{xy}'$ we have $y\ge x+a-(a_1+a_2-2)$. Moreover, since each of $A_1, A_2$ is a  $2$--progression, by the structure of $2(A_x)^-$ we have (see Figure \ref{fig:2progb} for an illustration)
\begin{equation}\label{eq:g(y)}
g(y)\le \left\{\begin{array}{ll} 1, &  x+a< y\le  2x\\ \lfloor h/2\rfloor +2, & y=x+a-h,\; 0\le h\le a_1+a_2-2 \end{array}\right. .
\end{equation}
\begin{figure}[h]
		\begin{center}
			\begin{tikzpicture}[scale=0.3]
			\foreach \i in {0,6,8,10,12,13,14,15,16,18,20}
			{
			\draw[fill] (\i,4) circle (6pt);	
			}
			\foreach \i in {1,2,3,4,5,7,9,11,17,19}
			{
			\draw (\i,4) circle (6pt);	
			}
			\node[left] at (-0.5,4) {$(A_x)^-$};
			\foreach \i in {0, 6, 8, 10, 12, 13, 14, 15, 16, 18, 19, 20, 21, 22, 23, 24, 25, 26,
27, 28, 29, 30, 31, 32, 33, 34, 35, 36, 38, 40}
			{
			\draw[fill] (\i,2) circle (6pt);	
			}
			\foreach \i in {1,2,3,4,5,7,9,11,17,37,39}
			{
			\draw (\i,2) circle (6pt);	
			}
			\node[left] at (-0.5,2) {$2(A_x)^-$};

\foreach \i in {0,6,8,10,12,13,14,15,16,18,20}
			{
			\draw[fill] (\i+26,0) circle (6pt);	
			}
			\foreach \i in {1,2,3,4,5,7,9,11,17,19}
			{
			\draw (\i+26,0) circle (6pt);	
			}
			\node[left] at (-0.5,0) {$y+(A_x)^-$};
			\draw (26,4)--(26,0);
			\draw (40,4)--(40,0);
			\draw (34,4)--(34,0);
			\node[above] at (26,4) {$y$};
			\node[above] at (40,4) {$2x$};
			\node[above] at (34,4) {$x+a$};
			\node[above] at (20,4) {$x$};
			
			\end{tikzpicture}
		\end{center}
		\caption{An illustration of the computation of $g(y)=|(y+(A_x)^-)\cap 2(A_x)^-|$.}\label{fig:2progb}
	\end{figure}
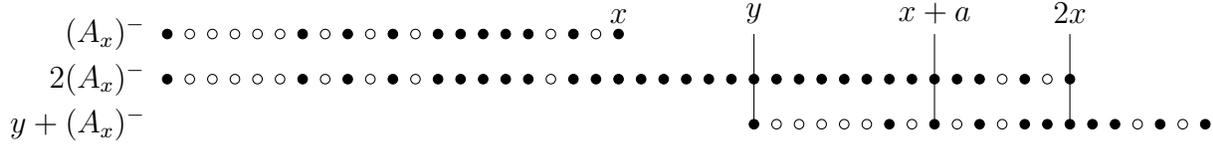

The proof of part (ii) is now completed in the same way as part (i).
 \end{enumerate}
\end{proof}

Our final Lemma concerns the sets of the form $A=A_1\circ P\circ A_2$ where $A_1$ and $A_2$ are $2$--progressions and $|P|=3$. This is a rather special case where the second operator $D_x$ introduced in Section \ref{sec:lb} will be used. We observe that in this case $A$ has an only odd number.

%
%
%

\begin{lemma}\label{lem:oneodd} Let $A$ be a chain with an only odd number $x$. The following hold:
	\begin{enumerate}
		\item[(i)]    $A=D_x(A')$ where $A'$ is a chain.
		\item[(ii)]  Every chain $B$ containing $A$ with $|2B|>3|B|-4$ has also an only odd number.
	\end{enumerate}
\end{lemma}

\begin{proof} (i)  The fact that $A=D_x(A')$ follows from the definition of the operator $D_x$. Let us show that $A'$ is a chain. 
	
	Since $A$ is a chain, it contains  three consecutive elements, one of them the only odd number in $A$. It follows that $A'$ contains two consecutive elements and hence $\gcd(A')=1$. Therefore $A'$ is in normal form. Moreover, if $B'$ is a chain with the same cardinality and doubling as $A'$ but wit larger volume, then $B=D(B')$ has the same cardinality and doubling as $A$ but with larger volume, contradicting that $A$ is extremal. It remains to show that $A'$ is a chain. Let $A_3\subset A_4 \cdots \subset A_k=A$ be the sequence of chains contained in $A$ (up to translation). We observe that $x\in A_3$. If $A'_i=D_x(A_i\setminus \{ x\})$ then $A'_4\subset \cdots \subset A'_k$ is a sequence of chains contained in $A'$.

(ii) In order to prove (ii), suppose that $A\cup \{y\}$ is a $1$--dimensional set with $y>\max(A)$ odd. We have
$$
|2(A\cup \{y\})|=|2A|+k+1-|(y+A)\cap 2A|.
$$
Since all elements in $y+A$ except $y+x$ are odd numbers, we have
\begin{equation}\label{eq:y}
|(y+A)\cap 2A|=|(y+A_0) \cap (x+A_0)|+|(y+x)\cap 2A_0|.
\end{equation}
We show that there is an even number $y'>y$ such that $A\cup \{y'\}$ is $1$--dimensional and has no larger doubling than $A\cup \{ y\}$, so that the latter set is not $1$--extremal. 

We consider two cases:

{\it Case 1.} $|(y+A_0)\cap (x+A_0)|\neq 0$.

Since  $A$ is  $1$--dimensional, $x$ must be involved in one elementary relation. Being the only odd number in $A$ the relation must be $2x=z+z'$ for some $z<z'\in A_0$. Let $\alpha=z'-x=x-z$ and $y'=y+\alpha$, an even number.

Since $|(y+A_0)\cap (x+A_0)|\neq 0$, we have $y+u=x+u'$ for some $u,u'\in A_0$ and hence 
$$
y'+u=z'+u'\in 2A_0,
$$
which shows that $A\cup \{y'\}$ is also $1$--dimensional. Moreover,
\begin{equation}\label{eq:y'}
|(y'+A)\cap 2A|\ge |(y'+A_0)\cap (z'+ A_0)|=|(y+A_0)\cap (x+A_0)|.
\end{equation}
If $|(x+y)\cap 2A_0|=0$ then \eqref{eq:y} and  \eqref{eq:y'} show that $A\cup \{y\}$ is not $1$--extremal.

If on the contrary $x+y=w+w'$ for some $w,w'\in A_0$ then $y'+z=(y+\alpha)+(x-\alpha)=w+w'\in |(y'+A)\cap 2A|$. We observe that, since $\alpha=z'-z$ is odd, $y'+z\not\in z'+A_0$ and therefore the inequality in \eqref{eq:y'} is strict. Again \eqref{eq:y} shows that $A\cup \{y\}$ is not $1$--extremal. 

%
%
%
%
%

{\it Case 2.} $|(y+A_0)\cap (x+A_0)|= 0$.

In this case we take $y'=2\max(A)>y$  which has doubling $|2A|+k$ as $A\cup \{y\}$, again contradicting that $A\cup \{ y\}$ is extremal.
\end{proof}

\section{Proof of main result}\label{sec:proof}

{\it Proof of Theorem \ref{thm:chain}} Suppose first that $A$ contains an only odd number. By Lemma \ref{lem:oneodd}(i) we can write $A=D_x(A')$ where $x$ is the odd number in $A$ and  $A'$ is a chain.
Hence we can write $A=\phi_{1}\cdots \phi_k(A_0)$ where each $\phi_i$ is of the form $D_{x_i}$ for some odd number $x_i$ and either $A_0$ has more than one odd number or $A_0\cong_F\{ 1,2,3\}$. In the latter case the Theorem holds with $B=\{ 1,2,3\}$. In what follows we assume that $A$ contains more than one odd number.

Let $A_3\subset A_4\subset \cdots \subset A_k=A$ be a chain  sequence of  $A$.
 Let $t$ be the largest integer $s$ for which $|2A_{s}|\le 3|A_s|-4$ and set $B=A_t$. By Lemma \ref{lem:oneodd}(ii) $B$ contains more than one odd number. 
 
 By Lemma \ref{cor:doubling}, we have $A_{t+1}\cong_F D(A_t)$ or $A_{t+1}\cong_F D^-(A_t)$ unless $B=B_0\circ P\circ B_2$ where $B_1$ and $B_2$ are $2$--progressions. If this is not the case then lemmas \ref{lem:unique} and \ref{lem:lunique} show that $A_{i}\cong_F D(A_{i-1})$ for each $t<i\le k$ proving the statement of the Theorem.
 
 The last case to consider is $B=B_0\circ P\circ B_2$ where $B_1$ and $B_2$ are $2$--progressions. Since $B$ contains more than one odd number we have $|P|\ge 4$. By Lemma \ref{lem:2prog2odd} we have  $A_{t+2}\cong_F D(A_{t+1})$. Lemmas \ref{lem:unique} and \ref{lem:lunique} show that  $A_{i}\cong_F D(A_{i-1})$ for each $t<i\le k$ proving the statement of the Theorem in this case. This completes the proof.

\section{Final Remarks}\label{sec:final}

The notion of chains for which Conjecture \ref{conj:main} is proved in this paper is a quite  natural one, and gives some evidence to the conjecture. Chains have the strong structure described in Theorem \ref{thm:chain}. Proving the conjecture for general sets has the difficulty of loosing these strong  structural properties along subsets, and additional techniques have to be used. So far we have only been able to address the case of doubling constant $c=3$,  the contents of a forthcoming paper,  which nevertheless opens a path which has been the object of many attempts for several decades. It seems unlikely that the conjecture can be proved for all values of the doubling constant $c$ up to $|A|/2$ as it has been for chains. It would be a significant breakthrough to know something about the structure of sets with doubling constant some growing function of $|A|$.

In the definition of chains we introduced the notion of $1$--extremal sets, sets with largest volume for given cardinality and doubling  among $1$--dimensional sets. A partial result towards the proof that extremal sets are in fact $1$--dimensional, as asserted in Conjecture \ref{conj:main}, is given in Freiman \cite{Freiman2014}. A natural question arises as if the conjectured maximum volume could be smaller if we restrict our sets to be $d$--dimensional with $d>1$. The following simple example shows that the lower bound $\mu (k,T)$ given in \eqref{eq:vmu} for the largest volume of a set with given cardinality $k$ and doubling $T$ is not far from the truth when restricted to $d$--dimensional sets. Let $A=A_1\cup \{e_2,\ldots ,e_d\}$ where $A_1\subset \R e_1$ is a $1$--extremal set and $e_1,\ldots ,e_d$ is the standard basis of the $d$--dimensional space $\R^d$. If $k=|A|$ and $T=|2A|$ then
$$
T=|2A_1|+k+1,
$$
while 
$$
vol(A)=vol(A_1)+(d-1)\ge \mu (k-(d-1), T-(k+1))+d,
$$
giving  a lower bound on the volume of  the $d$--dimensional set $A$ in terms of the function $\mu (k,T)$. 

The Freiman--Ruzsa theorem is instrumental in many applications. It would be interesting to feed the quantitative and structural information provided by Theorem \ref{thm:chain} to these applications to assess the relevance of the result.

\

 \end{document}